\documentclass[a4paper,12pt]{amsart}

\usepackage{amsmath, amssymb, amsthm, gensymb}
\usepackage{hyperref}
\usepackage{pgfplots, tikz, tikz-cd}
\usepackage{xfrac, relsize}
\usepackage{footnote}
\usepackage{subcaption}
\usepackage{enumitem}
\usepackage[utf8]{inputenc}
\usepackage{mathtools}
\usepackage{epigraph}
\usepackage{subcaption}
\usepackage{stmaryrd}
\usepackage{float}
\usepackage{mathrsfs}
\usepackage{graphicx}
\usepackage{import}
\usepackage{relsize}
\usepackage{leftindex}
\usepackage{calligra}
\usepackage{comment}
\usepackage[T1]{fontenc}
\usepackage{lmodern}
\usepackage{microtype}\DeclareMicrotypeAlias{ppl}{pplx}
\usepackage{todonotes}
\usepackage[backend=bibtex, style=alphabetic]{biblatex}
\calclayout

\hypersetup{
  colorlinks,
  citecolor = blue,
  urlcolor  = blue,
  linkcolor = blue,
}

\pgfplotsset{compat=1.16}
\usetikzlibrary{calc, arrows}

\newtheorem{definition}{Definition}[section]
\newtheorem*{question}{Question}
\newtheorem*{acknowledgements}{Acknowledgements}

\newtheorem{example}[definition]{Example}
\newtheorem{theorem}[definition]{Theorem}
\newtheorem*{theorem2}{Theorem}
\newtheorem{proposition}[definition]{Proposition}
\newtheorem{corollary}[definition]{Corollary}
\newtheorem{lemma}[definition]{Lemma}
\newtheorem{remark}[definition]{Remark}

\numberwithin{equation}{section}

\newtheorem{innercustomthm}{}
\newenvironment{customthm}[1]
  {\renewcommand\theinnercustomthm{#1}\innercustomthm}
  {\endinnercustomthm}

\newcommand{\defeq}{\vcentcolon=}

\newcommand{\restr}[2]{\ensuremath{\left.#1\right|_{#2}}}

\newcommand{\cat}[1]{\textbf{\scshape{#1}}}

\DeclareMathOperator{\im}{\operatorname{im}}

\DeclareMathOperator{\Spec}{Spec}
\DeclareMathOperator{\Proj}{Proj}

\newcommand*{\sheafext}{\mathcal{E}\text{\kern -2pt xt}}
\renewcommand{\leq}{\leqslant}
\renewcommand{\geq}{\geqslant}

\renewcommand{\chi}{\ensuremath\raisebox{\depth}{$\mathchar"11F$}}

\renewcommand{\_}{\rule{0.25cm}{0.4pt}}

\title[Finite group schemes as fundamental group schemes]{Finite group schemes as fundamental group schemes of smooth projective varieties}
\author{Gabriel Bassan}

\thanks{Sorbonne Université and Université Paris Cité, CNRS, IMJ-PRG, F-75005 Paris, France. 
E-mail: \texttt{bassan@imj-prg.fr}.}

\date{August 27, 2026}

\begin{document}

\begin{abstract}
In this paper we study the problem of realizing finite group schemes as fundamental group schemes of smooth projective varieties. We establish Bertini-type results and prove the triviality of fundamental group schemes of mildly singular complete intersections in projective space. With this in hand, we are able to go through a Godeaux--Serre construction and prove that, over an infinite perfect field $k$, any finite group scheme $G/k$ can be realized as the $S$-, extended Nori and Nori fundamental group schemes of a connected smooth projective variety of any dimension at least $\dim \text{Lie}(G)+2$.
\end{abstract}

\maketitle

\section*{Introduction}
A classical problem in geometry and topology is to determine which groups can appear as fundamental groups of spaces. Of course, the problem only becomes interesting when one imposes conditions on the spaces allowed. One classical condition is to look at fundamental groups of compact Kähler manifolds or smooth complex projective varieties. This question has spawned a lot of work and one can show that these groups satisfy strong geometric and representation-theoretic properties, but a complete classification is still out of reach (\cite{Amoros1996-im},\cite{Py25}). However, for the finite case it turns out there is no obstruction. In fact, we have the following theorem.

\begin{theorem2}\cite[Proposition 15]{Serre1958Topology}
Let $G$ be a finite group and $n\geq 1$ an integer. There exists a smooth connected complex projective variety of dimension $n$ with fundamental group isomorphic to $G$.
\end{theorem2}

The idea behind this is that of the Godeaux--Serre construction. One starts with a faithful representation $V$ of $G$ such that the non-free locus of the action of $G$ in $\mathbb{P}(V)$ has large codimension. By taking a suitable invariant smooth complete intersection contained in the free locus with the help of Bertini's theorem, one passes to the quotient and, through a short exact sequence relating the fundamental groups of a variety and its quotient by $G$, obtains a variety with fundamental group $G$.
\\
\\
One can ask the same question for algebraic versions of the fundamental group. For the étale fundamental group, over an algebraically closed field, the same Godeaux--Serre construction proves that any finite group is the étale fundamental group of a smooth projective variety over an algebraically closed field. For non-algebraically closed fields, there is also the input coming from the absolute Galois group of the field, but nonetheless, by modifying the Godeaux--Serre construction, one can prove that any continuous extension of the absolute Galois group by a finite group appears as the étale fundamental group of a smooth projective variety (see \cite{Rungtanapirom2018}).
\\
\\
Going further, one can ask the analogous question for Tannakian fundamental group schemes. The following question, due to V. B. Mehta, was communicated to the author by João Pedro dos Santos and is also recorded in the recent preprint of Li and Wang \cite{li2026varietiesprescribedfundamentalgroup}.

\begin{question}
Let $k$ be a field and $G/k$ be a finite group scheme. Is there a smooth connected projective variety $X/k$ such that its Tannakian group scheme (for some suitable choice of Tannakian category associated to $X$) $\pi_{1}^{\ast}(X,x)$ is isomorphic to $G$?
\end{question}

There is more than one choice for Tannakian fundamental group. Two natural choices are the Nori fundamental group (\cite{Nori}, \cite{Nori1982-tm}) and the $S$-fundamental group (\cite{AIF_2011__61_5_2077_0},\cite{Langer2012-ap}). Recently, in \cite{li2026varietiesprescribedfundamentalgroup}, Li and Wang settled this realization problem in the finite étale case by establishing a short exact sequence of fundamental groups for $G$-torsors and then using the same kind of Godeaux--Serre strategy. The purpose of this paper is to treat the general finite case. Our main result is the following.

\begin{customthm}{Theorem A}[Theorem \ref{fundamental group}]\label{A}
      Let $k$ be an infinite perfect field, $G/k$ a finite group scheme and $n\geq \dim(\text{Lie}(G))+2$ an integer. Then, there exists a connected smooth projective variety $X/k$ of dimension $n$ and $x\in X(k)$ such that
      \[
        \pi_{1}^{S}(X,x)\cong\pi_{1}^{EN}(X,x)\cong\pi_{1}^{N}(X,x)\cong G.
      \]
\end{customthm}

The main difficulty in the non-étale case is that, in contrast to the étale case, for a finite group scheme $G/k$ acting on projective space, one cannot find an invariant smooth complete intersection in which $G$ acts freely in general. Our strategy will be to instead find a singular invariant complete intersection with controlled singularities and extend the Godeaux--Serre construction in this case.
\\
\\
The first input we will need is a Bertini-type theorem which measures the singularities of a general complete intersection.

\begin{customthm}{Theorem B}[Theorem \ref{Bertini}]\label{B}
    Let $k$ be a field. Let $X/k$ be a smooth scheme over $k$ of pure dimension $n$ and $1\leq c\leq n$ be an integer. Let $L$ be a line bundle over $X$, $V$ be a $k$-vector space, $\varphi\colon V\to H^{0}(X,L)$ be a $k$-linear map and consider the $\mathcal{O}_{X}$-linear map
    \[
    p^{1}_{V}\colon V\otimes_{k}\mathcal{O}_{X}\to P^{1}(L).
    \]
    Suppose that
    \begin{enumerate}
        \item The image of $\varphi$ generates $L$;
        \item $p^{1}_{V}$ has constant rank $r>c$;
    \end{enumerate}
    Then for a generic $v\in V^{c}$, the zero locus $V(\varphi(v_{1}),\dots, \varphi(v_{c}))$ has dimension $n-c$ and its singular locus has dimension at most $n-r$.
\end{customthm}

The second input needed is that under mildly singular conditions, a complete intersection in projective space has trivial fundamental group. More precisely, we have the following.

\begin{customthm}{Theorem C}[Propositions \ref{intersecao completa normal é Nori simplesmente conexa} and \ref{intersecao completa normal é S simplesmente conexa}]\label{C}
    Let $k$ be a perfect field of characteristic $p>0$ and $n\geq 3$. Let $i\colon X=V(f_{1},\dots,f_{c})\hookrightarrow \mathbb{P}^{n}$ be a regular in codimension $2$ complete intersection of dimension $\geq 2$ and $x_{0}\in X(k)$. Then $\pi_{1}^{N}(X,x_{0})=\pi_{1}^{EN}(X,x_{0})=\pi_{1}^{S}(X,x_{0})=1$.
\end{customthm}

\subsection*{Organization of the paper}
\begin{itemize}
    \item In Section \ref{1} we go through preliminaries. In \ref{1.1} we recall some notions on reflexive sheaves. In \ref{1.2} we briefly go through local cohomology. In \ref{1.3} we recall some properties of complete intersections in projective space.
    \item Section \ref{2} contains the backbone of this manuscript. We prove some Bertini-type statements, in particular \ref{B}, which will allow us to refine the Godeaux--Serre construction in Section \ref{4}.
    \item Section \ref{3} is devoted to proving the triviality of fundamental groups of mildly singular complete intersections in projective space (\ref{C}). This section is where most of the preliminary notions introduced are used.
    \item In the final section, Section \ref{4}, we go through the Godeaux--Serre construction and prove the main result of the paper (\ref{A}).
\end{itemize}

\begin{acknowledgements}
I would like to thank my advisors Mathieu Florence and João Pedro dos Santos for their guidance and many helpful discussions. I am especially grateful to João Pedro dos Santos for bringing to my attention the question of V. B. Mehta that motivated this work.
\end{acknowledgements}

\section*{Notation and conventions}
All schemes are assumed to be quasi-compact and separated. A variety is a reduced scheme of finite type. For a scheme $X$ we denote its regular locus by $X_{\text{reg}}$ and its singular locus by $\text{Sing}(X)$.
\\
For a finite dimensional vector space $V$, we will denote its associated affine scheme $\Spec \text{Sym}(V^{\vee})$ simply by $V$. We also write $P(V)\defeq \Proj \text{Sym}(V^{\vee})$ for its projectivization. Given a line bundle $L$ on a scheme $X/k$, we denote its sheaf of order $1$ principal parts by $P^{1}_{X/k}(L)$ with the universal order $1$ differential operator
\[
p^{1}\colon L\to P^{1}_{X/k}(L).
\]
\\
If $G/k$ is a group scheme, we denote its Lie algebra by $\text{Lie}(G)$. By a $G$-torsor we always mean a fppf-torsor.
\section{Preliminaries}\label{1}
Before going into the main body of the manuscript, let me start by recalling some preliminary notions we will need.

\subsection{Reflexive sheaves on normal schemes}\label{1.1}

\begin{definition}
    Let $X$ be a noetherian scheme. A coherent module $\mathcal{F}$ is said to be reflexive if the natural morphism
    \[
    \mathcal{F}\to \mathcal{F}^{\vee\vee}
    \]
    to its double dual is an isomorphism.
\end{definition}

\begin{example}
    Since $\mathcal{O}_{X}$ is obviously reflexive, every locally free sheaf of finite rank is reflexive.
\end{example}

\begin{proposition}\label{propriedades feixes reflexivos}
    Let $X$ be an integral noetherian scheme.
    \begin{enumerate}
        \item The kernel of a morphism between reflexive sheaves is reflexive.
        \item Let $j\colon U\subseteq X$ be an open subscheme such that $\text{depth}(\mathcal{O}_{X,x})\geq 2$ for every $x\in X\setminus U$. Then $j_{\ast}$ and $j^{\ast}$ induce an equivalence between the categories of reflexive modules on $X$ and reflexive modules on $U$.
    \end{enumerate}
\end{proposition}
\begin{proof}
    The first statement is \cite[\href{https://stacks.math.columbia.edu/tag/0EBG}{Tag 0EBG}]{stacks-project} and the second one is \cite[\href{https://stacks.math.columbia.edu/tag/0EBJ}{Tag 0EBJ}]{stacks-project}. See also \cite[Section 1]{Hartshorne1980-ps}.
\end{proof}

\subsection{Local cohomology}\label{1.2}

\begin{definition}
    Let $X$ be a noetherian scheme and $i\colon Z\hookrightarrow X$ a closed subscheme. Consider the additive and left exact functor sending an $\mathcal{O}_{X}$-module $\mathcal{F}$ to
    \[
    \Gamma_{Z}(X,\mathcal{F})\defeq \{s\in \Gamma(X,\mathcal{F}) \mid \text{Supp}(s)\subseteq Z\}.
    \]
    We define the local cohomology groups $H^{i}_{Z}(X,\_)$ as the $i$-th derived functors of $\Gamma_{Z}(X,\_)$.
\end{definition}

\begin{proposition}\label{propriedades cohomologia local}
    Let $X$ be a noetherian scheme, $Z\subseteq X$ a closed subscheme and $U\defeq X\setminus Z$. Let $\mathcal{F}$ be a coherent module.
    \begin{enumerate}
        \item We have a long exact sequence
        \[
        0\to H^{0}_{Z}(X,\mathcal{F})\to H^{0}(X,\mathcal{F})\to H^{0}(U,\mathcal{F})\to H^{1}_{Z}(X,\mathcal{F})\to\cdots .
        \]
        \item $H^{i}_{Z}(X,\mathcal{F})=0$ for all $i<n$ if and only if $\text{depth}(\mathcal{F}_{x})\geq n$ for every $x\in Z$.
    \end{enumerate}
\end{proposition}
\begin{proof}
    The first statement is \cite[Exposé I, Corollaire 2.9]{SGA2}. The second statement is \cite[Exposé III, Lemme 3.1.(i) and Proposition 3.3.(iv)]{SGA2}.
\end{proof}

\subsection{Complete intersections in projective space}\label{1.3}

\begin{definition}
    Let $S/k$ be a smooth scheme. A closed subscheme $X\subseteq S$ of codimension $c$ is called a (global) complete intersection if there are $c$ effective Cartier divisors $D_{i}$ such that
    \[
    X=\bigcap_{i=1}^{c}D_{i}.
    \]
\end{definition}

\begin{definition}
    Let $X$ be a noetherian scheme and $k\geq 0$ an integer. We say that $X$ is $R_{k}$ (or regular in codimension $k$) if for every $x\in X$ such that $\dim(\mathcal{O}_{X,x})\leq k$, the local ring $\mathcal{O}_{X,x}$ is regular.
\end{definition}

\begin{proposition}\label{propriedades intersecoes completas}
    Let $i\colon X\hookrightarrow \mathbb{P}^{n}$ be a complete intersection of dimension $r\geq 1$. Then $X$ satisfies the following.
    \begin{enumerate}
        \item $X$ is a local complete intersection.
        \item $X$ is $R_{1}$ if and only if it is normal.
        \item For all $m\in \mathbb{Z}$, the morphisms
        \[
        H^{i}(\mathbb{P}^{n},\mathcal{O}(m))\to H^{i}(X,\mathcal{O}_{X}(m))
        \]
        are surjections if $i=0$ and isomorphisms if $0<i<r$. In particular, $X$ is geometrically connected.
        \item Let $\iota_{i}\colon X\hookrightarrow H_{i}$ be the inclusion. We have an isomorphism
        \[
        I_{X}/I_{X}^{2}\cong\bigoplus_{i=1}^{c}\iota_{i}^{\ast}(I_{H_{i}}/I_{H_{i}}^{2}).
        \]
    \end{enumerate}
\end{proposition}
\begin{proof}
    These are essentially in \cite[Chapter III, Ex. 5.5]{hartshorne}.
    \begin{enumerate}
        \item The local equations of the effective Cartier divisors define a regular sequence of length $c$ in each $\mathcal{O}_{X,x}$.
        \item Since $X$ is a local complete intersection, it is Cohen-Macaulay and the equivalence follows from Serre characterization of normal schemes (see \cite[\href{https://stacks.math.columbia.edu/tag/033P}{Tag 033P}]{stacks-project}).
        \item Let $d_{i}$ be the degree of $H_{i}$. Let $X_{0}\defeq \mathbb{P}^{n}$, $X_{k}\defeq H_{1}\cap\dots\cap H_{k}$ and let $i_{k}\colon X_{k+1}\hookrightarrow X_{k}$ be the inclusion. Then $X_{k+1}$ is an effective Cartier divisor in $X_{k}$ and a complete intersection in $\mathbb{P}^{n}$. We have an exact sequence
        \[
        0\to \mathcal{O}_{X_{k}}(m-d_{k+1})\to \mathcal{O}_{X_{k}}(m)\to i_{k,\ast}\mathcal{O}_{X_{k+1}}(m)\to 0.
        \]
        The claim then follows by induction on $c$ by using the long exact sequence in cohomology associated to the preceding exact sequence.
        \item This follows from \cite[Proposition 19.27]{Grtz2023}, but let me sketch the argument. Let $d_{i}$ be the degree of $H_{i}$ and $f_{i}\in H^{0}(\mathbb{P}^{n},\mathcal{O}(d_{i}))$ such that $H_{i}=V(f_{i})$. The claim follows by looking at the Koszul complex. Indeed, the beginning of the Koszul resolution of the ideal sheaf $I_{X}$ gives
        \[
        0\to K\to \bigoplus_{i,j}\mathcal{O}(-d_{i}-d_{j})\to \bigoplus_{i}\mathcal{O}(-d_{i})\to I_{X}\to 0,
        \]
        where the middle map is multiplication by $(f_{1},\cdots, f_{c})$. Restricting to $X$, the middle morphism vanishes and we get the isomorphism
        \[
        I_{X}/I_{X}^{2}\cong\bigoplus_{i=1}^{c}\mathcal{O}_{X}(-d_{i}).
        \]
    \end{enumerate}
\end{proof}

\section{Bertini's theorems}\label{2}
In this section, we prove some Bertini-type theorems.
\begin{definition}
    Let $V$ be a finite dimensional vector space over a field $k$. We say that a property holds for a generic $v\in V$ if it holds over a dense open set $U\subseteq V$ (as a $k$-scheme).
\end{definition}

\begin{definition}{(The universal section and zero locus)}\label{universal locus def}
Let $X/k$ be a finite type scheme of dimension $n$ and $1\leq c\leq n$ an integer. Let $L$ be a line bundle over $X$, $V$ be a finite dimensional $k$-vector space and $\varphi\colon V\to H^{0}(X,L)$ be a $k$-linear map such that its image generates $L$. We define the universal sections
    \[
    \sigma_{1},\dots, \sigma_{c}\in H^{0}(X\times V^{c}, p_{1}^{\ast}L)
    \]
    of $V$ as the image of the tautological $i$-th component sections 
    \[
    \tau_{1},\dots,\tau_{c}\in H^{0}(X\times V^{c}, V\otimes \mathcal{O}_{X\times V^{c}})
    \]
    by the morphism
    \[
    V\otimes \mathcal{O}_{X\times V^{c}}\to p_{1}^{\ast}L
    \]
    induced by $\varphi$. We also define the universal zero locus $\mathcal{Z}$ of $V$ as
    \[
    \mathcal{Z}\defeq V(\sigma_{1},\dots, \sigma_{c})\subseteq X\times V^{c}.
    \]
    We denote the maps $\mathcal{Z}\to X$ and $\mathcal{Z}\to V^{c}$ by $p_{X}$ and $p_{V}$.
\end{definition}

\begin{proposition}\label{universal locus prop}
Under the hypothesis of Definition \ref{universal locus def}, we have the following.
    \begin{enumerate}
        \item $\mathcal{Z}$ is isomorphic as an $X$-scheme to $\mathbb{V}_{X}(\ker(V\otimes\mathcal{O}_{X}\to L)^{c})$. In particular, it is a vector bundle of rank $c(\dim(V)-1)$ over $X$;
        \item The codimension of $\mathcal{Z}$ as a closed subscheme of $X\times V^{c}$ is equal to $c$;
        \item For $v=(v_{1},\dots, v_{c})\in V^{c}$, the fiber of the induced $k$-morphism $p_{V}\colon \mathcal{Z}\to V^{c}$ over $v$ is isomorphic to $V(\varphi(v_{1}),\dots, \varphi(v_{c}))\subseteq X$;
        \item If $p_{V}\colon \mathcal{Z}\to V^{c}$ is dominant, then for a generic $v\in V^{c}$ we have
        \[
        \dim(V(\varphi(v_{1}),\dots, \varphi(v_{c})))=n-c.
        \]
    \end{enumerate}
\end{proposition}
\begin{proof}
    \hfill
\begin{enumerate}
    \item First of all, observe that $\ker(V\otimes\mathcal{O}_{X}\to L)$ is indeed locally free of constant rank since $V\otimes\mathcal{O}_{X}\to L$ is surjective. Given an $X$-scheme $T\to X$, a morphism $T\to X\times V^{c}$ corresponds to giving $c$ sections $v_{1},\dots, v_{c}\in H^{0}(T,V\otimes\mathcal{O}_{T})$. Then $T\to X\times V^{c}$ factors through $\mathcal{Z}$ if and only if $v_{i}\in H^{0}(T,\ker(V\otimes\mathcal{O}_{T}\to L_{T}))$ for every $i$. Therefore, $\mathcal{Z}$ has the same functor of points as $\mathbb{V}_{X}(\ker(V\otimes\mathcal{O}_{X}\to L)^{c})$.
    \item This follows directly from (1).
    \item Given such an $v=(v_{1},\dots, v_{c})\in V^{c}$, it corresponds to a section $v\colon X\to X\times V^{c}$ of $p_{X}$ such that $v^{\ast}\sigma_{i}=v_{i}$. It is then clear that we have the commutative pullback diagram
    \[
    \begin{tikzcd}
    V(\varphi(v_{1}),\dots,\varphi(v_{c}))\ar[r]\ar[d, hookrightarrow]\arrow[dr, phantom, "\lrcorner", very near start] & \mathcal{Z}\ar[d, hookrightarrow]
    \\
    X\ar[r]\ar[d]\arrow[dr, phantom, "\lrcorner", very near start] & X\times V^{c}\ar[d]
    \\
    \Spec k\ar[r, "v"] & V^{c}.
    \end{tikzcd}
    \]
    \item If $p_{V}$ is dominant, the generic fiber of $\mathcal{Z}$ has dimension $\dim(\mathcal{Z})-c\dim(V)=n-c$ and the result follows.
\end{enumerate}
\end{proof}

We will be interested in understanding the singularities of the zero locus $i\colon Z=V(s_{1},\dots, s_{c})\hookrightarrow X$ of $c$ global sections $s_{i}$ of a line bundle. For that, we will need the following construction. Let $X/S$ be a smooth scheme over a base $S$, $L$ be a vector bundle over $X$ and $s_{1},\dots, s_{c}\in H^{0}(X,L)$. We can consider their image $p^{1}(s_{1}),\dots, p^{1}(s_{c})$ in the sheaf of first order principal parts $P_{X/S}^{1}(L)$. Upon restricting to $Z$, we have $i^{\ast}p^{1}(s_{i})\in \ker(i^{\ast}P^{1}_{X/S}(L)\to i^{\ast}L)$ for every $i$. Recall that we have an exact sequence
\[
0\to \Omega^{1}_{X/S}(L)\to P^{1}_{X/S}(L)\to L\to 0,
\]
and therefore, each $s_{i}$ induces a global section $ds_{i}\in H^{0}(Z, i^{\ast}\Omega^{1}_{X/S}(L))$. We then have the following lemma.

\begin{lemma}\label{bad locus characterization}
    Let $X/S$ be a smooth scheme of dimension $n$. Let $L$ be a line bundle over $X$ and $s_{1},\dots, s_{c}\in H^{0}(X,L)$. Suppose that $Z=V(s_{1},\dots, s_{c})$. Then $Z$ is smooth of codimension $c$ at a point $x\in Z$ if and only if the $ds_{1}(x),\dots, ds_{c}(x)$ are linearly independent in $\Omega^{1}_{X/S}(L)(x)$.
\end{lemma}
\begin{proof}
    Since the statement is local, we can suppose that $L$ is trivial and then this is the Jacobian criterion for smoothness.
\end{proof}

\begin{definition}{(Universal singular locus)}\label{singular locus def}
    Let $X/k$ be a smooth scheme of dimension $n$ and $1\leq c\leq n$ an integer. Let $L$ be a line bundle over $X$, $V$ be a finite dimensional $k$-vector space and $\varphi\colon V\to H^{0}(X,L)$ be a $k$-linear map such that its image generates $L$. Let $\mathcal{Z}$ be the universal zero locus and consider
    \[
    d\sigma_{1}\dots d\sigma_{c}\in H^{0}(\mathcal{Z},i_{\mathcal{Z}}^{\ast}(\Omega^{1}_{X\times V/V}(p_{X}^{\ast}L))).
    \]
    We define the universal singular locus $\mathcal{S}\subseteq \mathcal{Z}$ as
    \[
    \mathcal{S}\defeq V(d\sigma_{1}\wedge\cdots\wedge d\sigma_{c})\subseteq \mathcal{Z}.
    \]
    We will denote the induced morphisms $\mathcal{S}\to X$ and $\mathcal{S}\to V^{c}$ also by $p_{X}$ and $p_{V}$.
\end{definition}

\begin{proposition}\label{singular locus prop}
Under the hypothesis of Definition \ref{singular locus def}, we have the following.
    \begin{enumerate}
        \item Let $v=(v_{1},\dots, v_{c})\in V^{c}$. If $V(\varphi(v_{1}),\dots, \varphi(v_{c}))$ has codimension $c$, then the fiber $S_{v}$ over $v$ is isomorphic to the singular locus of $V(\varphi(v_{1}),\dots, \varphi(v_{c}))$.
        \item Suppose further that the induced $\mathcal{O}_{X}$-linear map
        \[
        p^{1}_{V}\colon V\otimes \mathcal{O}_{X}\to P^{1}_{X/k}(L)
        \]
        given by $v\otimes f\mapsto f\cdot P^{1}(\varphi(v))$ has constant rank $r>c$. Then the codimension of $\mathcal{S}$ as a closed subscheme of $\mathcal{Z}$ is equal to $r-c$.
    \end{enumerate}
\end{proposition}
\begin{proof}
    \hfill
    \begin{enumerate}
        \item This follows from Lemma \ref{bad locus characterization}.
        \item Since passing to a larger field extension does not change dimensions, we can suppose that $k$ is algebraically closed. Define the vector bundle of rank $r-1$,
        \[
        K\defeq \ker(\im p^{1}_{V}\to L).
        \]
        The morphism $p^{1}_{V}$ induces a surjection
        \[
        \psi\colon \ker(V\otimes \mathcal{O}_{X}\to L)\twoheadrightarrow K.
        \]
        Then the fiber of $\mathcal{S}$ over a closed point $x\in X(k)$ is given by
        \[
        \{(v_{1},\dots, v_{c})\in \ker(V\to L(x))\mid \psi_{x}(v_{1})\wedge\dots\wedge\psi_{x}(v_{c})=0\}
        \]
        In other words, the fiber over $x$ is given by the preimage of the variety $D_{c}$ of linearly dependent $c$-tuples in $K(x)^{c}$ by $\psi^{c}_{x}$. $D_{c}$ is isomorphic to a determinantal variety of $(r-1)\times c$ matrices of rank at most $c-1$ and therefore, has dimension $(c-1)r$. We then have that
        \begin{align*}
        \dim \mathcal{S}_{x} &=\dim (\psi^{c}_{x})^{-1}(D_{c})=c\dim\ker(\psi_{x})+\dim D_{c}   
        \\
        &=c(\dim(V)-r)+(c-1)r=c\dim(V)-r.
        \end{align*}
        Since $p_{X}\colon \mathcal{S}\to X$ is surjective, we get that for any closed point $x$,
        \[
        \dim \mathcal{S}=\dim(X)+\dim(\mathcal{S}_{x})=n+c\dim(V)-r
        \]
        and thus
        \[
        \text{codim}_{\mathcal{Z}}(\mathcal{S})=(n+c(\dim(V)-1))-(n+c\dim(V)-r)=r-c.
        \]    
    \end{enumerate}
\end{proof}

Our main Bertini-type result is the following.

\begin{theorem}\label{Bertini}
    Let $k$ be a field. Let $X/k$ be a smooth scheme over $k$ of pure dimension $n$ and $1\leq c\leq n$ be an integer. Let $L$ be a line bundle over $X$, $V$ be a $k$-vector space, $\varphi\colon V\to H^{0}(X,L)$ be a $k$-linear map and consider the $\mathcal{O}_{X}$-linear map
    \[
    p^{1}_{V}\colon V\otimes_{k}\mathcal{O}_{X}\to P^{1}(L).
    \]
    Suppose that
    \begin{enumerate}
        \item The image of $\varphi$ generates $L$;
        \item $p^{1}_{V}$ has constant rank $r>c$;
    \end{enumerate}
    Then for a generic $v\in V^{c}$, the zero locus $V(\varphi(v_{1}),\dots, \varphi(v_{c}))$ has dimension $n-c$ and its singular locus has dimension at most $n-r$.
\end{theorem}
\begin{proof}
    First of all, we can suppose that $k$ is algebraically closed since passing to the larger field extension does not affect dimensions or smoothness. In this case, the projection from the universal zero locus $p_{V}\colon \mathcal{Z}\to V^{c}$ is dominant. Indeed, it is enough to show that its image contains a nonempty open set. For this, let
    \[
    K\defeq \ker(\im p^{1}_{V}\to L)\subseteq \im p^{1}_{V}.
    \]
    Observe that for a fixed closed point $x\in X(k)$, since $\dim K(x)=r-1\geq c$, we can find $c$ linearly independent vectors $w_{1},\dots, w_{c}\in K(x)$. Since $V\to \im p^{1}_{V}$ is surjective, we can lift each $w_{i}$ to vectors $v_{i}\in V$ which satisfy 
    \begin{enumerate}
        \item $\varphi(v_{i})(x)=0$ for every $i$ and
        \item $d\varphi(v_{1})(x),\dots, d\varphi(v_{c})(x)$ are linearly independent in $\Omega^{1}_{X/k}(x)$.
    \end{enumerate}
    Let $v\defeq (v_{1},\dots, v_{c})$ Now, let $U$ be an open neighborhood of $x$ where $L$ is trivialized, let $e\in \Gamma(U\times V^{c},p_{1}^{\ast}L)$ be a generator and write $\sigma_{i}=f_{i}\cdot e$ with $f_{i}\in \Gamma(X\times V^{c},\mathcal{O}_{X\times V^{c}})$. Under the identification
    \[
    \Omega^{1}_{X\times V^{c}/V^{c}}\xrightarrow{\sim} p_{1}^{\ast}\Omega^{1}_{X/k}
    \]
    the vector $df_{i}(x,v)\in \Omega^{1}_{X\times V^{c}/V^{c}}(x,v)$ is sent to $ds_{i}(x)\in p_{1}^{\ast}\Omega^{1}_{X/k}$. Thus, the $df_{i}(x,v)$'s are linearly independent in $\Omega^{1}_{X\times V^{c}/V^{c}}(x,v)$ and, $p_{V}$ is smooth at $(x,v)$. Since a smooth morphism is open, we conclude that the image of $p_{V}$ contains a nonempty open set. In particular, by Proposition \ref{universal locus prop}, we have 
    \[
    \dim(\mathcal{Z}_{v})=n-c
    \]
    for a generic $v\in V^{c}$. By Proposition $\ref{singular locus prop}$, the fiber $\mathcal{S}_{v}$ of $\mathcal{S}$ over a point $v\in V^{c}$ such that $\mathcal{Z}_{v}$ has dimension $n-c$ is exactly its singular locus. We would like to calculate its dimension. For that, let $\{\mathcal{S}_{i}\}_{i\in I}$ be the (finitely many) irreducible components of $\mathcal{S}$. Let $I_{1}\subseteq I$ be the subset of indices such that 
    \[
    \overline{p_{V}(\mathcal{S}_{i})}=V^{c}
    \]
    and $I_{2}=I\setminus I_{1}$. For each $i\in I_{1}$, the morphism $p_{V,i}\colon \mathcal{S}_{i}\to V^{c}$ is dominant and therefore, the generic fiber satisfies
    \[
    \dim p_{V,i}^{-1}(v)=\dim \mathcal{S}_{i}-\dim V^{c}\leq \dim \mathcal{S}-\dim V^{c}=n-r.
    \]
    Let $U_{1}\subseteq V^{c}$ be a dense open subset for which the preceding inequality holds for every $i\in I_{1}$ and
    \[
    U_{2}=V^{c}\setminus(\bigcup_{i\in I_{2}} \overline{p_{V,i}(\mathcal{S}_{i})}).
    \]
    Then for every closed point $v\in U_{1}\cap U_{2}$ we have
    \[
    \dim p_{V}^{-1}(v)= \max_{i\in I}\dim p_{V,i}^{-1}(v)\leq n-r.
    \]
    Restricting ourselves to the dense open of $V$ where $\mathcal{Z}_{v}$ has dimension $n-c$, for a generic $s\in V^{c}$,
    \[
    \text{dim}(\text{Sing}(\mathcal{Z}_{v}))\leq n-r
    \]
    and we are done.
\end{proof}

\begin{remark}
The preceding theorem in the case $c=1$ is essentially contained in the proof of \cite[Theorem 2.2]{Biswas2024}. Their theorem allows singular $X$ and drops the constant rank hypothesis. On the other hand, they enunciate it for $k$ algebraically closed and $X/k$ projective. Since the constant rank on a smooth scheme is enough for our purposes while the non-projectivity of $X/k$ is also needed, we decided to give a self-contained proof.
\end{remark}

We can also rewrite the preceding theorem as a statement about morphisms to projective space instead of linear systems.

\begin{proposition}\label{Bertini 2}
    Let $k$ be a field. Let $X/k$ be a smooth scheme over $k$ of pure dimension $n$ and $1\leq c\leq n$ be an integer. Let
    \[
    f\colon X\to \mathbb{P}^{d}
    \]
    be a morphism to projective space. Suppose that the pullback map
    \[
    f^{\ast}\Omega^{1}_{\mathbb{P}^{d}/k}\to \Omega^{1}_{X/k}
    \]
    has constant rank $\rho\geq c$. Then for $c$ generic hyperplanes $H_{1},\dots, H_{c}$ in $\mathbb{P}^{d}$, $f^{-1}(H_{1}\cap\dots\cap H_{c})$ is a complete intersection of dimension $n-c$ and its singular locus has dimension at most $n-\rho-1$.
\end{proposition}
\begin{proof}
    A simple calculation shows that the map
    \[
    H^{0}(\mathbb{P}^{d},\mathcal{O}(1))\otimes \mathcal{O}_{X}\to P^{1}_{X}(f^{\ast}\mathcal{O}_{X}(1))
    \]
    has constant rank $r=\rho+1$. The proposition is then equivalent to Theorem \ref{Bertini}.
\end{proof}

\begin{lemma}\label{unramified}
    Let $k$ be a field, $X/k$ a $k$-scheme of finite type and $f\colon X\to\mathbb{P}^{d}$ a morphism. Let $L\defeq f^{\ast}\mathcal{O}_{\mathbb{P}^{d}}(1)$, denote $V\defeq H^{0}(\mathbb{P}^{d},\mathcal{O}_{\mathbb{P}^{d}}(1))$ and consider the pullback map $\varphi\defeq f^{\ast}\colon V\to H^{0}(X,L)$. Then $f$ is unramified if and only if the induced map
    \[
    p^{1}_{V}\colon V\otimes_{k}\mathcal{O}_{X}\to P_{X}^{1}(L)
    \]
    is surjective.
\end{lemma}
\begin{proof}
    The map $p^{1}_{V}$ factors through $f^{\ast} P_{\mathbb{P}^{d}}^{1}(1)\to P^{1}_{X}(L)$. Since $\mathcal{O}(1)$ is very ample, the map
    \[
    V\otimes \mathcal{O}_{\mathbb{P}^{d}}\to P_{\mathbb{P}^{d}}^{1}(1)
    \]
    is surjective and therefore,
    \[
    V\otimes \mathcal{O}_{X}\to f^{\ast}P_{\mathbb{P}^{d}}^{1}(1)
    \]
    is surjective. Thus, surjectivity of $p^{1}_{V}$ is equivalent to surjectivity of $f\ast P_{\mathbb{P}^{d}}^{1}(1)\to P^{1}_{X}(L)$. Locally, this map is given by
    \[
    p^{1}_{f}\colon f^{\ast} P_{\mathbb{P}^{d}}^{1}\to P^{1}_{X}.
    \]
    Thus it is enough to prove the surjectivity of $p^{1}_{f}$. Now, we have the following diagram
    \[
    \begin{tikzcd}
    0\ar[r] & f^{\ast}\Omega^{1}_{\mathbb{P}^{d}}\ar[r]\ar[d, "f^{\ast}"] & f^{\ast}P^{1}_{\mathbb{P}^{d}}\ar[r]\ar[d, "p^{1}_{f}"] & \mathcal{O}_{X}\ar[r]\ar[d, equal] & 0
    \\
    0\ar[r] & \Omega^{1}_{X}\ar[r] & P^{1}_{X}\ar[r] & \mathcal{O}_{X}\ar[r] & 0.
    \end{tikzcd}
    \]
    A diagram chase shows that $P^{1}_{f}$ is surjective if and only if $f^{\ast}$ is surjective, i.e. if and only if $\Omega^{1}_{X/\mathbb{P}^{d}}=0$.
\end{proof}

Therefore, in the case where $r=n+1$, we get Bertini's classical theorem as a corollary.

\begin{corollary}{(Bertini's theorem)}\label{Bertini usual}
    Let $k$ be a field. Let $X/k$ be a smooth scheme over $k$ of pure dimension $n$ and fix an integer $1\leq c\leq n$. Let 
    \[
    f\colon X\to \mathbb{P}^{d}
    \]
    be an unramified morphism to projective space. Then for a generic $(H_{1},\dots, H_{c})\in H^{0}(\mathbb{P}^{d},\mathcal{O}(1))^{c}$, $f^{-1}(H_{1}\cap\dots\cap H_{c})$ is a smooth complete intersection of dimension $n-c$.
\end{corollary}

We would like to also prove a pointed version Theorem \ref{Bertini}.

\begin{proposition}\label{Bertini pontuado}
    Let $k$ be a field. Let $X/k$ be a smooth scheme over $k$ of pure dimension $n$ and $1\leq c\leq n$ be an integer. Let $L$ be a line bundle over $X$, $V$ be a $k$-vector space, $\varphi\colon V\to H^{0}(X,L)$ be a $k$-linear map and consider the $\mathcal{O}_{X}$-linear map
    \[
    p^{1}_{V}\colon V\otimes\mathcal{O}_{X}\to P^{1}_{X/k}(L).
    \]
    Fix $x_{0}\in X(k)$ and let $W\defeq \ker(V\to L(x_{0}))$. Suppose that
    \begin{enumerate}
        \item The image of $\varphi$ generates $L$;
        \item $p^{1}_{V}$ has constant rank $r>c$;
        \item For every closed point $x$ in $U\defeq X\setminus \{x_{0}\}$,
        \[
        \ker p^{1}_{V}(x)\not\subseteq W\otimes k(x).
        \]
    \end{enumerate}
    Then for a generic $v\in W^{c}$, the zero locus $V(\varphi(v_{1}),\dots, \varphi(v_{c}))$ contains $x_{0}$, has dimension $n-c$ and its singular locus has dimension at most $n-r$.
\end{proposition}
\begin{proof}
    First of all, condition $(3)$ implies that the morphism
    \[
    p^{1}_{W}\colon W\otimes \mathcal{O}_{X}\to P^{1}_{U/k}(L)
    \]
    has constant rank $r$ over $U$. Indeed, for a closed point $x\in U$, since $\ker p^{1}_{V}(x)\not\subseteq W\otimes k(x)$, and $W\otimes k(x)$ has dimension $\dim(V)-1$, the image of the restriction of $p^{1}_{V}(x)$ to $W\otimes k(x)$, which is $p^{1}_{W}(x)$, is the same as the image of $p^{1}_{V}$. In particular, $W$ generates $L$ over $U$. By Theorem \ref{Bertini} applied to $W\to H^{0}(X,L)$ over $U$, we get that for a generic $v\in W^{c}$, the zero locus $V(\varphi(v_{1}),\dots, \varphi(v_{c}))\cap U\subseteq U$ has dimension $n-c$ and its singular locus at $U$ has dimension at most $n-r$. Thus, the only point which can cause problems is $x_{0}$. We will prove that generically, this point is smooth of codimension $c$. At $x_{0}$, since $W=\ker(V\to L(x_{0}))$, the image of $p^{1}_{W}$ at $x_{0}$ is equal to
    \[
    \ker(\im p^{1}_{V}(x_{0})\to L(x_{0}))
    \]
    and therefore has dimension $r-1$. This means that the map
    \[
    W\to \Omega^{1}_{X/k}(L)(x_{0})
    \]
    sending $v$ to $d\varphi(v)(x_{0})$ has rank $r-1$. Since $r-1\geq c$, we get that for a generic $v\in W^{c}$, the $d\varphi(v_{i})(x_{0})$'s are linearly independent in $\Omega^{1}_{X/k}(L)(x_{0})$ and by Lemma \ref{bad locus characterization}, that $V(\varphi(v_{1}),\dots, \varphi(v_{c}))$ is smooth of codimension $c$ at $x_{0}$.
\end{proof}

We also state and prove some auxiliary results which will be useful.

\begin{lemma}\label{ignorar fechado}
    Let $k$ be a field. Let $X/k$ be a scheme of finite type over $k$ of pure dimension $n$ and let $U\subseteq X$ be an open subscheme. Let $L$ be a line bundle over $X$, $V$ be a $k$-vector space and $\varphi\colon V\to H^{0}(X,L)$ be a $k$-linear map such that its image generates $L$. Then for a generic $v=(v_{1},\dots, v_{c})\in V^{c}$, we have 
    \[
    \dim(V(\varphi(v_{1}),\dots, \varphi(v_{c}))\cap (X\setminus U))\leq \dim (X\setminus U)-c.
    \]
\end{lemma}
\begin{proof}
    Once again, we can suppose that $k$ is algebraically closed. Let $i\colon X\setminus U\hookrightarrow X$ be the closed immersion. Consider the universal zero locus $\mathcal{Z}$ of $V$ over $X\setminus U$. Then $\mathcal{T}\to X\setminus U$ is a vector bundle of rank $c(\dim(V)-1)$ and therefore, of dimension $\dim (X\setminus U)+c(\dim(V)-1)$. Now, consider the projection $\mathcal{T}\to V^{c}$. Observe that for a closed point $v=(v_{1},\dots, v_{c})\in V^{c}$, its fiber is exactly $\mathcal{Z}_{v}\cap (X\setminus U)$. By a similar argument as in the proof of Theorem \ref{Bertini} applied to the projection $\mathcal{T}\to V^{c}$, we get that for a generic $v\in V^{c}$,
    \[
    \dim(\mathcal{Z}_{v}\cap (X\setminus U))\leq \dim (X\setminus U)-c.
    \]
\end{proof}

We also have a pointed version of this lemma.

\begin{lemma}\label{ignorar fechado pontuado}
Let $k$ be a field. Let $X/k$ be a scheme of finite type over $k$ of pure dimension $n$ and let $U\subseteq X$ be an open subscheme. Let $L$ be a line bundle over $X$, $V$ be a finite dimensional $k$-vector space and $\varphi\colon V\to H^{0}(X,L)$ be a $k$-linear map such that its image generates $L$. Fix $x_{0}\in U(k)$ and let $W\defeq \ker(V\to L(x_{0}))$. Suppose that the image of
\[
W\to H^{0}(X,L)
\]
generates $L$ over $X\setminus U$. Then for a generic $v\in W^{c}$, we have 
\[
\dim(V(\varphi(v_{1}),\dots, \varphi(v_{c}))\cap (X\setminus U))\leq \dim (X\setminus U)-c.
\]
\end{lemma}
\begin{proof}
    The exact same proof as in the unpointed version works.
\end{proof}

Before finishing the section, we prove a technical lemma which will be needed later.

\begin{lemma}\label{potencia boa}
Let $k$ be a field. Let $X/k$ be a smooth scheme and fix $x_{0}\in X(k)$. Let $L$ be a line bundle over $X$, $V$ be a finite dimensional $k$-vector space and $\varphi\colon V\to H^{0}(X,L)$ be a $k$-linear map such that
\begin{enumerate}
    \item The image of $\varphi$ generates $L$;
    \item $p^{1}_{V}$ has constant rank $r>1$;
    \item $\varphi$ separates geometric points, i.e for every pair of points $x_{1}\neq x_{2}\in X(\overline{k})$, there exists $v\in V\otimes \overline{k}$ such that $\varphi(v)(x_{1})\neq 0$ and $\varphi(v)(x_{2})=0$.
\end{enumerate}
 Let $m>1$ be an integer. Let
\[
W_{m}\defeq \ker(V^{\otimes m}\to L^{\otimes m}(x_{0})).
\]
Then
\begin{enumerate}
    \item The image of $\varphi^{m}\colon V^{\otimes m}\to H^{0}(X,L^{\otimes m})$ generates $L^{\otimes m}$;
    \item $p^{1}_{V^{\otimes m}}$ has constant rank $r$;
    \item $\varphi^{m}(W_{m})$ generates $L^{\otimes m}$ on $X\setminus \{x_{0}\}$;
    \item For every closed point $x\neq x_{0}$,
    \[
    \ker p^{1}_{V^{\otimes m}}(x)\not\subseteq W_{m}\otimes k(x).
    \]
\end{enumerate}
\end{lemma}
\begin{proof}
    Passing to the algebraic closure, we can suppose that $k$ is algebraically closed.
    \begin{enumerate}
        \item This is clear.
        \item Let $x$ be a closed point and $v_{0}\in V$ be such that $\varphi(v_{0})(x)\neq 0$. Let $s_{0}\defeq \varphi(v_{0})$. Then locally around $x$ $s_{0}$ trivializes $L$ and we have an isomorphism
        \[
        F\colon P^{1}_{X/k}(L)\xrightarrow{\sim} P^{1}_{X/k}(L^{\otimes m})
        \]
        given by $p^{1}(s)\mapsto p^{1}(s\otimes s_{0}^{\otimes m-1})$. I will prove that $I\defeq \im p^{1}_{V}$ is taken to $I_{m}\defeq \im p^{1}_{V^{\otimes m}}$ under this isomorphism. Indeed, for $v\in V$
        \[
        F(p^{1}_{V}(v))=F(p^{1}(\varphi(v)))=p^{1}(\varphi(v)\otimes s_{0}^{\otimes m-1})=p^{1}_{V^{\otimes m}}(v\otimes v_{0}^{\otimes m-1})\in I_{m}
        \]
        Since $I$ is generated as an $\mathcal{O}_{X}$-module by the $p^{1}_{V}(v)$'s, $F(I)\subseteq I_{m}$. On the other hand, it is enough to show that for $v_{1},\dots, v_{m}\in V$, we have
        \[
        p^{1}_{V^{\otimes m}}(v_{1}\otimes\dots\otimes v_{m})\in F(I).
        \]
        Let $\varphi(v_{i})=f_{i}s_{0}$ near $x$. Then
        \[
        p^{1}_{V^{\otimes m}}(v_{1}\otimes\dots\otimes v_{m})=p^{1}(f_{1}\cdots f_{m}s_{0}^{\otimes m}).
        \]
        Since $p^{1}$ is a first order differential operator, we get from the Leibniz
        \[
        p^{1}_{V^{\otimes m}}(v_{1}\otimes\dots\otimes v_{m})=\sum_{i=1}^{m}\left(\prod_{j\neq i}f_{j}\right)P^{1}(f_{i}s_{0}^{\otimes m})-(m-1)\left(\prod_{i}f_{i}\right)p^{1}(s_{0}^{\otimes m}).
        \]
        Since $p^{1}(f_{i}s_{0}^{\otimes m})=p^{1}(s_{i}s_{0}^{\otimes m-1})$ and $p^{1}(s_{0}^{\otimes m})$ belong in $F(I)$, we are done.
        \item Given a closed point $x\in X\setminus\{x_{0}\}$, let $v_{0}\in V$ such that $\varphi(v_{0})(x_{0})=0$ and $\varphi(v_{0})(x)\neq 0$. Then $v_{0}^{\otimes m}\in W_{m}$ and $\varphi^{\otimes m}(x)\neq 0$. We conclude that $\varphi^{m}(W_{m})$ generates $L^{\otimes m}$ on $X\setminus \{x_{0}\}$.
        \item Let $v_{0}\in V$ be such that $\varphi(v_{0})(x_0)\neq 0$ and $\varphi(v_{0})(x)= 0$. By the Leibniz rule, we have $p^{1}_{V^{\otimes m}}(v_{0}^{\otimes m})=p^{1}(\varphi(v_{0}^{\otimes m}))(x)=0$ and thus $v_{0}^{\otimes m}\in \ker P^{1}_{V^{\otimes m}}(x)$ but $v_{0}^{\otimes m}\not\in W_{m}$.
    \end{enumerate}
\end{proof}

\section{Fundamental groups of complete intersections in projective space}\label{3}

Let $k$ be a field. Let $X/k$ be a proper connected and reduced scheme and $x\in X(k)$. Recall that we can associate to it different variations of a fundamental group.
\begin{itemize}
    \item The Nori fundamental group scheme $\pi^{N}_{1}(X,x)$ (\cite{Nori},\cite{Nori1982-tm}), i.e the Tannakian group of the category $\cat{EF}(X)$ of essentially finite vector bundles over $X$.
    \item The extended Nori fundamental group scheme $\pi^{EN}_{1}(X,x)$ (\cite{Ota17},\cite{Amr20}), i.e the Tannakian group of the saturation $\cat{EN}(X)$ of $\cat{EF}(X)$ (meaning vector bundles which have a filtration with successive quotients in $\cat{EF}(X)$).
    \item The $S$-fundamental group scheme $\pi^{S}_{1}(X,x)$ (\cite{AIF_2011__61_5_2077_0},\cite{Langer2012-ap}), i.e the Tannakian group of the category $\cat{NF}(X)$ of numerically flat vector bundles over $X$.
    \item The local fundamental group $\pi^{loc}_{1}(X,x)$ (\cite{Mehta2008-pi}), i.e the maximal local quotient of $\pi^{N}_{1}(X,x)$. Its category of representations is isomorphic to the category of Frobenius trivial vector bundles (meaning vector bundles trivialized by some power of Frobenius).
\end{itemize}
We always have inclusions
\[
\cat{EF}(X)\subseteq \cat{EN}(X)\subseteq \cat{NF}(X)
\]
which induce faithfully flat morphisms
\[
\pi^{S}_{1}(X,x)\twoheadrightarrow \pi^{EN}_{1}(X,x)\twoheadrightarrow \pi^{N}_{1}(X,x).
\]

In this section, we prove the vanishing of fundamental groups of complete intersections in projective space under mild regularity conditions. We start by establishing some results about the vanishing of global differential $1$-forms.

\begin{lemma}\label{conormal}
    Let $k$ be a perfect field. Let $i\colon Z\hookrightarrow X$ be a regular closed immersion with $X/k$ a smooth scheme and $Z/k$ a normal scheme. Then the conormal sequence
    \[
    I_{Z}/I^{2}_{Z}\to i^{\ast}\Omega^{1}_{X/k}\to \Omega^{1}_{Z/k}\to 0
    \]
    is also left exact.
\end{lemma}
\begin{proof}
    Since $i$ is regular and $X/k$ is smooth, $I_{Z}/I^{2}_{Z}$ and $i^{\ast}\Omega^{1}_{X/k}$ are locally free of finite rank and therefore, reflexive modules. Therefore, the kernel of $I_{Z}/I^{2}_{Z}\to i^{\ast}\Omega^{1}_{X/k}$ is also reflexive. Since $Z/k$ is normal, it is $R_{1}$ and thus a reflexive module on $Z$ is zero if it is when restricted to $Z_{\text{reg}}$. The proposition then follows since $I_{Z}/I^{2}_{Z}\to i^{\ast}\Omega^{1}_{X/k}$ is injective when restricted to the regular locus.
\end{proof}

\begin{lemma}\label{Intersecao completa é simplesmente conexa}
    Let $k$ be a perfect field and $n\geq 3$. Let $i\colon X=V(f_{1},\dots,f_{c})\hookrightarrow \mathbb{P}^{n}$ be a normal complete intersection of dimension $\geq 2$ with $f_{i}\in H^{0}(\mathbb{P}^{n},\mathcal{O}_{\mathbb{P}^{n}}(d_{i}))$. Then the group $H^{0}(X,\Omega^{1}_{X/k})$ vanishes.
\end{lemma}
\begin{proof}
    By Lemma \ref{conormal}, the conormal sequence of $i$,
    \[
    0\to I_{X}/I_{X}^{2}\to i^{\ast}\Omega^{1}_{\mathbb{P}^{n}/k}\to \Omega^{1}_{X/k}\to 0
    \]
    is exact. By Proposition \ref{propriedades intersecoes completas}, we have
    \[
    I_{X}/I_{X}^{2}\cong \bigoplus_{i=1}^{c}\mathcal{O}_{X}(-d_{i}).
    \]
    Looking at the long exact sequence in cohomology, to prove the claim, it is enough to show that
    \begin{enumerate}
        \item $H^{1}(X,\mathcal{O}_{X}(-d_{i}))=0$ for every $i$,
        \item $H^{0}(X,i^{\ast}\Omega^{1}_{\mathbb{P}^{n}/k})=0$.
    \end{enumerate}
    The first assertion follows from \ref{propriedades intersecoes completas}. For the second one, look at the Euler sequence
    \[
    0\to \Omega^{1}_{\mathbb{P}^{n}/k}\to \mathcal{O}_{\mathbb{P}^{n}}(-1)^{n+1}\to \mathcal{O}_{\mathbb{P}^{n}}\to 0.
    \]
    Pulling back by $i$, we get a short exact sequence
    \[
    0\to i^{\ast}\Omega^{1}_{\mathbb{P}^{n}/k}\to \mathcal{O}_{X}(-1)^{n+1}\to \mathcal{O}_{X}\to 0.
    \]
    Thus we have an injection $H^{0}(X,i^{\ast}\Omega^{1}_{\mathbb{P}^{n}/k})\hookrightarrow H^{0}(X,\mathcal{O}_{X}(-1)^{n+1})$ and it is enough to show that $H^{0}(X,\mathcal{O}_{X}(-1))=0$. This follows again from \ref{propriedades intersecoes completas}.
\end{proof}

\begin{lemma}\label{reg}
    Let $k$ be a perfect field and $n\geq 3$. Let $i\colon X=V(f_{1},\dots,f_{c})\hookrightarrow \mathbb{P}^{n}$ be a $R_{2}$ complete intersection of dimension $\geq 2$. Then the group $H^{0}(X_{\text{reg}},\Omega^{1}_{X_{\text{reg}}/k})$ vanishes.
\end{lemma}
\begin{proof}
Let $Z\defeq \text{Sing}(X)$. We have a long exact sequence in local cohomology
\[
0\to H^{0}_{Z}(X,\Omega^{1}_{X/k})\to H^{0}(X,\Omega^{1}_{X/k})\to H^{0}(X_{\text{reg}},\Omega^{1}_{X_{\text{reg}}/k})\to H^{1}_{Z}(X,\Omega^{1}_{X/k})\to\cdots.
\]
Since $X$ is $R_{2}$, it is normal and by Lemma \ref{Intersecao completa é simplesmente conexa}, $H^{0}(X,\Omega^{1}_{X/k})=0$. Thus, it is enough to prove that 
\[
H^{1}_{Z}(X,\Omega^{1}_{X/k})=0.
\]
This follows if we prove that $\text{depth}(\Omega^{1}_{X,x})\geq 2$ for all $x\in Z$. Since $X$ is a normal complete intersection, we have an exact sequence
\[
0\to I_{X}/I_{X}^{2}\to i^{\ast}\Omega^1_{\mathbb{P}^{n}/k}\to \Omega^{1}_{X/k}\to 0.
\]
By \cite[\href{https://stacks.math.columbia.edu/tag/00LX}{Tag 00LX}]{stacks-project}, we have
\[
\text{depth}(\Omega^{1}_{X/k,x})\geq \min\{\text{depth}(I_{X,x}/I_{X,x}^{2})-1,\text{depth}(\Omega^1_{\mathbb{P}^{n}/k,x})\}.
\]
Since both $I_{X}/I_{X}^{2}$ and $i^{\ast}\Omega^{1}_{\mathbb{P}^{n}/k}$ are locally free of finite rank, they have depth at $x$ equal to the depth of $\mathcal{O}_{X,x}$. Since $X$ is an $R_{2}$ complete intersection (in particular, $\mathcal{O}_{X,x}$ is Cohen-Macaulay for every $x\in X$) and $x\in \text{Sing}(X)$, we find that
\[
\text{depth}(\mathcal{O}_{X,x})=\dim\mathcal{O}_{X,x}\geq 3.
\]
We conclude that
\[
\text{depth}(\Omega^{1}_{X/k,x})\geq 2.
\]
\end{proof}

We are now able to prove results about the fundamental group of a complete intersection in projective space.

\begin{proposition}\label{intersecao completa normal é Nori simplesmente conexa}
Let $k$ be a perfect field and $n\geq 3$. Let $i\colon X=V(f_{1},\dots,f_{c})\hookrightarrow \mathbb{P}^{n}$ be an $R_{2}$ complete intersection of dimension $\geq 2$ and $x_{0}\in X(k)$. Then $\pi_{1}^{N}(X,x_{0})=1$.
\end{proposition}
\begin{proof}
    First of all, by \cite[Proposition 5.14.(5)]{li2026basechangefundamentalgroup}, we can suppose that $k$ is algebraically closed. By Lefschetz theorem for the étale fundamental group (\cite[Exposé X, Exemple 2.2 and Théorème 3.10]{SGA2}), we have
    \[
    \pi_{1}^{ét}(X,x_{0})\cong \pi_{1}^{ét}(\mathbb{P}^{n},x_{0})=1.
    \] 
    If the characteristic of $k$ is zero, then $\pi_{1}^{N}(X,x_{0})=\pi_{1}^{ét}(X,x_{0})=1$. If the characteristic of $k$ is positive, we then have $\pi_{1}^{N}(X,x_{0})=\pi_{1}^{loc}(X,x_{0})$. Since
    \[
    \cat{rep}_{k}(\pi_{1}^{loc}(X,x_{0}))\xrightarrow{\sim} \cat{FT}(X),
    \]
    where $\cat{FT}(X)$ is the category of Frobenius trivial vector bundles, i.e vector bundles $E$ such that $(F_{X}^{n})^{\ast}E$ is trivial for some $n\geq 0$, it is enough to prove that any Frobenius trivial vector bundle is trivial. Given such a vector bundle $E$ over $X$, its restriction to $X_{\text{reg}}$ is also Frobenius trivial. By Lemma \ref{reg}, we have $H^{0}(X_{\text{reg}},\Omega^{1}_{X_{\text{reg}}})=0$. Since $k$ is perfect, we have an isomorphism $X_{\text{reg}}\cong X_{\text{reg}}^{(p)}$ and thus, any connection on $\restr{(F_{X}^{n})^{\ast}E}{X_{\text{reg}}}$ is trivial. By Cartier descent (see \cite[Theorem 5.1]{Katz}), the restriction of $E$ to the regular locus of $X$ is trivial. Since $X$ is normal, $E$ is reflexive and $\text{codim}_{X}\text{Sing}(X)\geq 2$, we conclude that $E$ is trivial.
\end{proof}

\begin{proposition}\label{intersecao completa normal é S simplesmente conexa}
Let $k$ be a perfect field of characteristic $p>0$ and $n\geq 3$. Let $i\colon X=V(f_{1},\dots,f_{c})\hookrightarrow \mathbb{P}^{n}$ be an $R_{2}$ complete intersection of dimension $\geq 2$ and $x_{0}\in X(k)$. Then $\pi_{1}^{S}(X,x_{0})=\pi_{1}^{EN}(X,x_{0})=1$.
\end{proposition}
\begin{proof}
    It is enough to prove the claim for $\pi_{1}^{S}(X,x_{0})$. We can assume that $k$ is algebraically closed by \cite[Proposition 5.8.(4)]{li2026basechangefundamentalgroup}. Let us then suppose that the characteristic of $k$ is positive. We prove the statement by induction on the dimension of $X$. If $\dim(X)=2$, then $X/k$ is smooth. By \cite[Corollary 8.3]{Langer2012-ap}, we have $\pi_{1}^{S}(X,x_{0})\cong \pi_{1}^{N}(X,x_{0})$. By Proposition \ref{intersecao completa normal é Nori simplesmente conexa}, we conclude that $\pi_{1}^{S}(X,x_{0})=1$. Suppose now that $\dim(X)\geq 3$, then by Proposition \ref{Bertini 2} and Lemma \ref{ignorar fechado}, intersecting $X$ with a general hyperplane $H$ in $\mathbb{P}^{n}$ gives us another complete intersection $Y\defeq X\cap H$ of dimension $\dim(X)-1$ such that $X_{\text{reg}}\cap H$ is smooth and $\text{Sing}(X)\cap H$ has dimension at most $\dim(\text{Sing}(X))-1$. In particular, $Y$ is $R_{2}$. Since $k$ is algebraically closed, we can choose $y_{0}\in Y(k)$. By induction, we have $\pi_{1}^{S}(Y,y_{0})=1$ and thus, that any numerically flat bundle on $Y$ is trivial. Now, given a numerically flat bundle $E$ over $X$, its restriction to $Y$ is still numerically flat. Thus, it is trivial when restricted to $Y$. By \cite[Theorem 2.1]{ASENS_2014__47_4_833_0}, $E$ is Frobenius trivial. In particular, this implies that $\pi_{1}^{S}(X,x_{0})\cong \pi_{1}^{N}(X,x_{0})$ which is trivial by Proposition \ref{intersecao completa normal é Nori simplesmente conexa}.
\end{proof}

\section{Godeaux--Serre varieties and fundamental groups}\label{4}

Recall that a Godeaux--Serre variety is a smooth projective variety $X$ which is a quotient of a complete intersection in projective space by a free action of a finite group scheme $G$. Of course this complete intersection will then be a $G$-torsor over $X$, but usually we do not have control over the singularities of it if $G$ is not smooth. We will give a construction of Godeaux--Serre varieties in such a way that this complete intersection has controlled singularities.

\begin{lemma}\label{Bertini torsor}
    Let $k$ be a field. Let $G/k$ be a group scheme and $p\colon P\to X$ be a $G$-torsor with $X/k$ and $P/k$ smooth and connected of dimension $n$ and $m$ respectively. Let $g\colon X\to \mathbb{P}^{d}$ be an unramified morphism and denote $f\defeq g\circ p$. Consider the line bundle $L\defeq f^{\ast}\mathcal{O}(1)$ and the morphism
    \[
    \varphi\colon V\defeq H^{0}(\mathbb{P}^{d},\mathcal{O}(1))\to H^{0}(P,L).
    \]
    Then, $p^{1}_{V}\colon V\otimes \mathcal{O}_{P}\to P^{1}_{P/k}(L)$ has constant rank equal to $m-\dim\text{Lie}(G)+1$.
\end{lemma}
\begin{proof}
    After trivializing $p\colon P\to X$, the sheaf $\Omega^{1}_{P/X}$ becomes isomorphic to $\text{Lie}(G)^{\vee}\otimes \mathcal{O}_{P}$. Therefore, $\Omega^{1}_{P/X}$ is locally free of rank equal to $\dim\text{Lie}(G)$. Therefore, the pullback map $p^{\ast}\Omega^{1}_{X/k}\to \Omega^{1}_{P/k}$ has constant rank equal to $m-\dim\text{Lie}(G)$. Since $g$ is unramified, we get that $f^{\ast}\Omega^{1}_{\mathbb{P}^{d}/k}\to \Omega^{1}_{P/k}$ also has constant rank $m-\dim\text{Lie}(G)$. Therefore, the map
    \[
    f^{\ast}P^{1}_{\mathbb{P}^{d}/k}\to P^{1}_{P/k}
    \]
    has constant rank equal to $m-\dim\text{Lie}(G)+1$. Twisting by $L$ we get that
    \[
    f^{\ast}(P^{1}_{\mathbb{P}^{d}/k}(\mathcal{O}(1)))\to P^{1}_{P/k}(L)
    \]
    has constant rank equal to $m-\dim\text{Lie}(G)+1$. By Lemma \ref{unramified}, $V\otimes \mathcal{O}_{P}\to f^{\ast}(P^{1}_{\mathbb{P}^{d}/k}(\mathcal{O}(1)))$ is surjective and we are done.
\end{proof}

\begin{proposition}\label{Godeaux--Serre normal}
    Let $k$ be an infinite field, $G/k$ be a finite group scheme and $n\geq \dim(\text{Lie}(G))+2$ an integer. Then there exists $N>0$ and an $R_{2}$ complete intersection $Y\subseteq \mathbb{P}^{N}$ with $Y(k)\neq \varnothing$ on which $G$ acts freely and such that the quotient $Y/G$ is a smooth projective scheme of dimension $n$ over $k$. Moreover, if $G/k$ is étale, then $Y$ is smooth.
\end{proposition}
\begin{proof}
    Let $R\defeq \mathcal{O}(G)$ be the regular representation of $G$ and $d\defeq \dim R$. Consider the action of $G$ on $\mathbb{P}(R^{r})\cong \mathbb{P}^{N}$ for some $r>0$, where $N=dr-1$. The free-locus $U\subseteq \mathbb{P}(R^{r})$ of this action is an open subscheme such that $\dim(\mathbb{P}(R^{r})\setminus U)\leq N-r$ (See \cite[Lemma 4.2.2]{AST_1979__64__87_0} and \cite[Lemma 3.6.]{kothari2024arbitrarilylargejumpsrham}). Now, the quotient $X\defeq \mathbb{P}^{N}/G$ is a projective variety (see \cite[Exposé V, Théorème 4.1.]{SGA3}) and the quotient $X^{\circ}\defeq U/G$ is a smooth open subscheme of $\mathbb{P}^{N}/G$. Moreover, the quotient $p\colon \mathbb{P}^{N}\to X$ is a $G$-torsor when restricted to $U$. Since $k$ is infinite, we can fix a point $y_{0}\in U(k)$. Let $x_{0}\defeq p(y_{0})\in X^{\circ}(k)$ and fix a closed immersion $i\colon X\hookrightarrow \mathbb{P}^{M}$. Let $V\defeq H^{0}(\mathbb{P}^{M},\mathcal{O}(1))$ and $L\defeq i^{\ast}\mathcal{O}(1)$. Take $r>0$ sufficiently large so that $r\geq n+1$ and $c=N-n>0$.
    \begin{enumerate}
        \item Consider the morphism $\varphi_{X^{\circ}}\colon V\to H^{0}(X^{\circ},L)$. Since $i$ is an immersion, the image of $V\otimes \mathcal{O}_{X^{\circ}}$ in $P^{1}_{X^{\circ}/k}(L)$ has maximal rank $N+1$. By Lemma \ref{potencia boa}, upon substituting $V$ and $L$ by some power, we can suppose that $V$ and $W\defeq \ker(V\to L(x_{0}))$ satisfy the hypothesis of Proposition \ref{Bertini pontuado} on $X^{\circ}$. Thus, we get that for a generic $v\in W^{c}$, the closed subscheme $Z\defeq V(\varphi_{X}(v_{1}),\dots,\varphi_{X}(v_{c}))$ of $X$ contains $x_{0}$, has dimension $N-c=n$ and the singular locus of $Z\cap X^{\circ}$ has dimension $N-(N+1)=-1$, i.e $Z\cap X^{\circ}$ is smooth. 
        \item Consider then $\varphi_{X}\colon V\to H^{0}(X,L)$. By applying Lemma \ref{ignorar fechado pontuado} to $\varphi_{X}$ and the open immersion $X^{\circ}\subseteq X$, we have for a generic $v\in W^{c}$ that
        \begin{align*}
        \dim(Z\cap (X\setminus X^{\circ})) &\leq \dim(X\setminus X^{\circ})-c=\dim(\mathbb{P}^{N}\setminus U)-c
        \\
        &\leq N-r-c=n-r\leq -1.
        \end{align*}
        Then $Z\subseteq X^{\circ}$. So $Z$ is a smooth complete intersection of dimension $n$ contained in $X^{\circ}$ such that $x_{0}\in Z(k)$.
        \item Consider the morphism $\psi\colon V\to H^{0}(\mathbb{P}^{N},p^{\ast}L)$. Let $Y\defeq p^{-1}(Z)=V(\psi(v_{1}),\dots, \psi(v_{c}))$. By construction, $y_{0}\in Y(k)$, $G$ acts freely on $Y$ and $Y$ is a complete intersection of dimension $n$ in $\mathbb{P}^{N}$. It remains to show that $Y$ is $R_{2}$. If $\dim\text{Lie}(G)=0$, $p$ is étale and therefore, $Y$ is also smooth. So suppose that $\dim\text{Lie}(G)>0$. Let $F\defeq p^{-1}(x_{0})$ and consider the morphisms 
        \[
        \psi_{U}\colon V\to H^{0}(U,p^{\ast}L)
        \]
        and
        \[
        \psi_{U\setminus F}\colon W\to H^{0}(U\setminus F,p^{\ast}L)
        \]
        This morphism satisfies the hypothesis of Theorem \ref{Bertini}. We would like to prove that
        \[
        W\otimes \mathcal{O}_{U\setminus F}\to P^{1}_{(U\setminus F)/k}(p^{\ast}L)
        \] 
        has constant rank equal to $N-\dim\text{Lie}(G)+1$. By Lemma $\ref{Bertini torsor}$,
        \[
        V\otimes \mathcal{O}_{U}\to P^{1}_{U/k}(p^{\ast}L)
        \]
        has constant rank $N-\dim\text{Lie}(G)+1$. As in the proof of Proposition \ref{Bertini pontuado},
        \[
        W\otimes \mathcal{O}_{U\setminus F}\to P^{1}_{(U\setminus F)/k}(p^{\ast}L)
        \] 
        also has constant rank $N-\dim\text{Lie}(G)+1$. Thus, for a generic $v\in W^{c}$, we have
        \[
        \dim\text{Sing}(Y\setminus F)\leq N-(N-\dim\text{Lie}(G)+1)=\dim\text{Lie}(G)-1.
        \]
        Since $\text{Sing}(Y)\subseteq \text{Sing}(Y\setminus F)\cup F$, we have
        \[
        \dim\text{Sing}(Y)\leq \max\{\dim\text{Lie}(G)-1,\dim(F)\}=\dim\text{Lie}(G)-1.
        \]
        We then obtain
        \[
        \text{codim}_{Y}(\text{Sing}(Y))\geq n-\dim\text{Lie}(G)+1\geq 3.
        \]
        So $Y$ is an $R_{2}$ complete intersection containing $y_{0}\in Y(k)$ in which $G$ acts freely with quotient $Z$.
    \end{enumerate}
    Since $k$ is infinite, we can find a $v\in W^{c}$ satisfying all of these generic properties and we are done.
\end{proof}

With this construction in hand, we are able to prove the main realization result for fundamental group schemes.

\begin{theorem}\label{fundamental group}
    Let $k$ be an infinite perfect field, $G/k$ a finite group scheme and $n\geq \dim(\text{Lie}(G))+2$ an integer. Then there exists a smooth projective connected variety $X/k$ of dimension $n$ and $x\in X(k)$ such that
    \[
    \pi_{1}^{S}(X,x)\cong\pi_{1}^{EN}(X,x)\cong\pi_{1}^{N}(X,x)\cong G.
    \]
\end{theorem}
\begin{proof}
    Since we have
    \[
    \pi^{S}_{1}(X,x)\twoheadrightarrow \pi^{EN}_{1}(X,x)\twoheadrightarrow \pi^{N}_{1}(X,x),
    \]
    if $\pi^{S}_{1}(X,x)\cong \pi^{N}_{1}(X,x)\cong G$ with $G/k$ finite, then $\pi^{EN}_{1}(X,x)\cong G$. Thus, it is enough to prove the statement for the $S$ and Nori fundamental group schemes. By \cite[Corollary I and Theorem IV]{Antei2019-og} (see also \cite[Theorem 4.7.]{li2026varietiesprescribedfundamentalgroup}), for a $G$-torsor $Y\to X$, with $X$ and $Y$ geometrically reduced connected proper schemes over $k$, and a point $y\in Y(k)$ with image $x\in X(k)$, we have the short exact sequence 
    \[
    1\to \pi_{1}^{\ast}(Y,y)\to \pi_{1}^{\ast}(X,x)\to G\to 1,
    \]
    for $\ast\in\{S, N\}$. Let $Y\to X$ and $y\in Y(k)$ be as given by Proposition \ref{Godeaux--Serre normal}. We will show that $\pi_{1}^{\ast}(Y,y)=1$ for $\ast\in\{S, N\}$. By \cite[Propositions 5.8.(4) and 5.14.(5)]{li2026basechangefundamentalgroup}, it is enough to show this after passing to the algebraic closure. If the characteristic of $k$ is zero, then $G/k$ is étale and $Y$ is smooth. Since $Y$ is a complete intersection of dimension $\geq 2$ in projective space, we have $\pi^{ét}_{1}(Y,y)=1$ by Grothendieck-Lefschetz. Thus, we have
    \[
    \pi_{1}^{S}(Y,y)=\pi_{1}^{N}(Y,y)=1
    \]
    by \cite[Section 2]{EsnaultMehta2011}. If the characteristic of $k$ is positive, by Propositions \ref{intersecao completa normal é Nori simplesmente conexa} and \ref{intersecao completa normal é S simplesmente conexa}, we have
    \[
    \pi_{1}^{S}(Y,y)=\pi_{1}^{N}(Y,y)=1.
    \]
    The theorem then follows from the short exact sequence.
\end{proof}

\begin{remark}
    The infinitude assumption on the field in Proposition \ref{Godeaux--Serre normal} (and hence in Theorem \ref{fundamental group}) is used to choose rational points on dense open sets of the parameter spaces. Nevertheless, we expect that Poonen's Bertini theorems over finite fields (\cite{Poonen2004-ei}) can be adapted to our setting, obtaining the analogous results of Section \ref{4} for finite fields. We do not pursue this direction here.
\end{remark}

\printbibliography
\end{document}